\documentclass[oneside,english]{amsart}
\usepackage[T1]{fontenc}
\usepackage{amsthm}
\usepackage{amssymb}
\usepackage{esint}
\usepackage[all]{xy}
\usepackage{lmodern}
\usepackage{tikz}

\makeatletter

\numberwithin{equation}{section}
\numberwithin{figure}{section}
\theoremstyle{plain}
\newtheorem{thm}{\protect\theoremname}[section]
  \theoremstyle{plain}
  \newtheorem{prop}[thm]{\protect\propositionname}
  \theoremstyle{definition}
  \newtheorem{defn}[thm]{\protect\definitionname}
  \theoremstyle{remark}
  \newtheorem{rem}[thm]{\protect\remarkname}
  \theoremstyle{plain}
  \newtheorem{lem}[thm]{\protect\lemmaname}
  \theoremstyle{definition}
  
  \newtheorem{cor}[thm]{\protect\corollaryname}
  \theoremstyle{plain}

\makeatother

\usepackage{babel}
  \providecommand{\definitionname}{Definition}
  \providecommand{\examplename}{Example}
  \providecommand{\lemmaname}{Lemma}
  \providecommand{\propositionname}{Proposition}
  \providecommand{\remarkname}{Remark}
  \providecommand{\corollaryname}{Corollary}
\providecommand{\theoremname}{Theorem}

\begin{document}

\title{Simplicial localization of homotopy algebras over a prop}

\address{Laboratoire Paul Painlevé, Université de Lille 1, Cité Scientifique,
59655 Villeneuve d'Ascq Cedex, France}
\email{Sinan.Yalin@math.univ-lille1.fr}

\author{Sinan Yalin}
\begin{abstract}
We prove that a weak equivalence between two cofibrant (colored) props in chain complexes induces
a Dwyer-Kan equivalence between the simplicial localizations of the associated categories of algebras.
This homotopy invariance under base change implies
that the homotopy category of homotopy algebras over a prop $P$ does not depend on
the choice of a cofibrant resolution of $P$, and gives thus a coherence to the notion of algebra up to homotopy
in this setting. The result is established more generally for algebras in combinatorial monoidal dg categories.


\textit{AMS} : 18G55 ; 18D50 ; 18D10 ; 55U10.

\end{abstract}
\maketitle

\tableofcontents{}

\section{Introduction}

Props are combinatorial devices introduced by MacLane in \cite{MLa} in order to parametrize operations with multiple inputs
and outputs. The usual categories of bialgebras, like the classical associative bialgebras, the Lie bialgebras, ...,
which are defined by combining an algebra and a coalgebra structure constrained by a distribution relation (see \cite{Lod} for many examples),
can be modelled by props. The category of Frobenius bialgebras gives another example of a bialgebra structure
modelled by a prop (of a more general form).
Under suitable conditions on the ambient category, the category of props admits
a model category structure \cite{Fre}.
Let $P$ be a prop. We define a $P$-algebra up to homotopy, or homotopy $P$-algebra, as an algebra over a cofibrant resolution $Q(P)$ of $P$.
When going from a $P$-algebra structure to a $Q(P)$-algebra structure, the relations defining the $P$-algebra structure
are relaxed and are only satisfied up to a (whole) set of (coherent) homotopies.
This definition involves the choice of a cofibrant resolution, so a natural coherence request is to ask for
the homotopy theory of homotopy algebras to not depend on this choice.

Homotopy algebras structures appear in various context, in particular when transferring an algebra structure along
a weak equivalence, or when studying realization problems, that is, given a $P$-algebra structure on the homology
$H_*X$, the finer $Q(P)$-algebra structures on $X$ inducing this $P$-algebra in homology.
A concrete example is the Frobenius bialgebra structure on the cohomology of an oriented compact
manifold with coefficients in a field. Other kinds of algebraic structures coded
by props are for instance Lie bialgebras or conformal field theories.

\medskip

\noindent
\textit{The main result.}
Let $(C,W)$ be a pair of categories such that $W$ is a subcategory of weak equivalences of $C$ in the sense of \cite{DK2}.
Such a pair is called a relative category in \cite{BK1}. To every relative category one can associate its simplicial
localization \cite{DK1}, denoted by $L(C,W)$. It is a simplicial category which encodes the whole homotopy theory of $(C,W)$.
The category of connected components $\pi_0 L(C,W)$ gives the homotopy category $C[W^{-1}]$.
Our main contribution in this paper is to prove an optimal homotopical invariance result for the homotopy theory of
algebras over props, which provides a definitive answer to our coherence problem.
We work with props in the category $Ch_{\mathbb{K}}$ of chain complexes over a field $\mathbb{K}$ (dg props for short).
Let $Ch_{\mathbb{K}}^P$ denote the category of algebras over a dg prop $P$, and
$wCh_{\mathbb{K}}^P$ its subcategory of morphisms of $P$-algebras which are quasi-isomorphisms in $Ch_{\mathbb{K}}$.
The main result in this paper reads:
\begin{thm}
A weak equivalence $\varphi:P\stackrel{\sim}{\rightarrow}Q$ between cofibrant (colored) dg props induces a Dwyer-Kan equivalence
of simplicial localizations
\[
L(\varphi^*):L(Ch_{\mathbb{K}}^Q,wCh_{\mathbb{K}}^Q)\stackrel{\cong}{\rightarrow} L(Ch_{\mathbb{K}}^P,wCh_{\mathbb{K}}^P).
\]
\end{thm}
Thus the homotopy theory (both the primary and secondary information)
of homotopy $P$-algebras does not depend on the choice of a cofibrant resolution of $P$.
Theorem 1.1 implies in particular the existence of an equivalence of homotopy categories
\[
Ch_{\mathbb{K}}^Q[(wCh_{\mathbb{K}}^Q)^{-1}] \cong Ch_{\mathbb{K}}^P[(wCh_{\mathbb{K}}^P)^{-1}].
\]
We actually need the result of Theorem 1.1 to establish this more basic relation.
Indeed, other classical methods fail. For instance, since algebras over a cofibrant prop do not form a model category
in general, their homotopy category is difficult to handle directly and we cannot use the machinery of
Quillen equivalences.
We combine simplicial localization techniques with the properties of $\infty$-categories, and elaborate on methods introduced in \cite{Yal1}
about the invariance of classification spaces in the propic setting.

\medskip

\noindent
\textit{Relation with the previous works.} Homotopy invariance for algebraic structures has a long history, going back to \cite{BV}.
There are two kinds of homotopy invariance results that one wants to establish, giving rise to two parallel stories of results.
The first sort consists in proving that homotopy algebraic structures
can be transfered over homotopy equivalences in the ground category. This goes back to Boardman and Vogt \cite{BV} in the topological setting.
It was established by numerous authors for different special cases of algebras, before the general result for homotopy algebras over dg operads by Markl \cite{Mar}.
Fresse finally proved it for props in monoidal model categories \cite{Fre}, a result extended by Johnson and Yau to the colored case \cite{JY}.

The second sort is a homotopy invariance under base change, that is an invariance of the homotopy category of homotopy algebras,
up to equivalence, under weak equivalences of cofibrant operads. It goes back to the preprint of
Getzler and Jones \cite{GJ} for dg operads and the thesis of Rezk \cite{Rez1} for simplicial operads. The general case of algebras
over cofibrant operads in monoidal model categories has been treated by Berger and Moerdijk \cite{BM1}, then extended to the
colored case \cite{BM2}. It has been extended to colored props by Johnson and Yau \cite{JY} under the strong assumption of the existence
of a free algebra functor. This covers the case of cartesian categories but excludes, for instance, chain complexes.

This paper brings a final answer for homotopy invariance under base change of algebras over cofibrant colored props in dg categories,
with new methods different from those of the preceding works (which rely on the existence of a model structure on algebras).

\medskip

\noindent
\textit{Organization of the paper.}

Section 2 consists in brief recollections about symmetric monoidal
model categories over a base category and props and algebras over a prop in this setting.
Section 3 is the heart of this paper. It has been divided in three steps, each one using properties
of $(\infty,1)$-categories to finally reduce the problem to a comparison of classification diagrams.
This comparison is proved by using a homotopy invariance result obtained in \cite{Yal1}.

\medskip

\noindent
\textit{Acknowledgements}
I would like to thank Denis-Charles Cisinski for decisive remarks on results of \cite{Yal1} which lead to this paper.

\section{Recollections on props, algebras and homotopy}

We use the following relative version of the notion of a symmetric monoidal category:
\begin{defn}
Let $\mathcal{C}$ be a symmetric monoidal category. A \emph{symmetric monoidal
category over $\mathcal{C}$} is a symmetric monoidal category $(\mathcal{E},\otimes_{\mathcal{E}},1_{\mathcal{E}})$
endowed with a symmetric monoidal functor $\eta:\mathcal{C}\rightarrow\mathcal{E}$,
that is, an object under $\mathcal{C}$ in the $2$-category of symmetric monoidal categories.

This defines on $\mathcal{E}$ an external tensor product $\otimes :\mathcal{C}\times\mathcal{E}\rightarrow\mathcal{E}$
by $C\otimes X = \eta(C)\otimes_{\mathcal{E}} X$ for every $C\in\mathcal{C}$ and $X\in\mathcal{E}$.
This external tensor product is equiped with the following unit, associativity and symmetry isomorphisms natural
in each argument:

(1) $\forall X\in\mathcal{E},1_{\mathcal{C}}\otimes X\cong X$,

(2) $\forall X\in\mathcal{E},\forall C,D\in\mathcal{C},(C\otimes D)\otimes X\cong C\otimes (D\otimes X)$,

(3) $\forall C\in\mathcal{C},\forall X,Y\in\mathcal{E},C\otimes (X\otimes Y)\cong(C\otimes X)\otimes Y\cong X\otimes(C\otimes Y)$.

The coherence constraints of these natural isomorphisms
(associativity pentagons, symmetry hexagons and unit triangles which mix both internal and external tensor products)
come from the symmetric monoidal structure of the functor $\eta$.

We assume throughout the paper that all small limits and small
colimits exist in $\mathcal{C}$ and $\mathcal{E}$, and that each of these categories admits an internal hom bifunctor. We suppose moreover the existence of an
external hom bifunctor $Hom_{\mathcal{E}}(-,-):\mathcal{E}^{op}\times\mathcal{E}\rightarrow\mathcal{C}$
satisfying an adjunction relation
\[
\forall C\in\mathcal{C},\forall X,Y\in\mathcal{E},Mor_{\mathcal{E}}(C\otimes X,Y)\cong Mor_{\mathcal{C}}(C,Hom_{\mathcal{E}}(X,Y)).
\]
\end{defn}
\begin{rem}
The assumption above is equivalent to assuming $\mathcal{E}$ is enriched and tensored over $\mathcal{C}$).
\end{rem}

When $\mathcal{C}$ is a symmetric monoidal category equiped with a model structure, we require
the following compatibility axioms:
\begin{defn}
(1) A symmetric monoidal model category is a symmetric monoidal category
$\mathcal{C}$ endowed with a model category structure such that
the following axioms hold:

\textbf{MM0.} The unit object $1_{\mathcal{C}}$ of $\mathcal{C}$ is cofibrant.

\textbf{MM1.} The pushout-product $(i_{*},j_{*}):X\otimes T\oplus_{X\otimes Z}Y\otimes Z\rightarrow Y\otimes T$
of cofibrations $i:X\rightarrowtail Y$ and $j:Z\rightarrowtail T$ is a cofibration
which is also acyclic as soon as $i$ or $j$ is so.

(2) Suppose that $\mathcal{C}$ is a symmetric monoidal model category.
A symmetric monoidal category $\mathcal{E}$ over $\mathcal{C}$ is
a symmetric monoidal model category over $\mathcal{C}$ if the axiom
MM0 holds and the axiom MM1 holds for both the internal and external tensor
products of $\mathcal{E}$.
\end{defn}

Axiom MM0 ensures the existence of a unit for the monoidal structure of the homotopy category.
Axiom MM1 provides the necessary assumptions to make the tensor product a Quillen bifunctor.
Let us mention that axiom MM0 is weakened in the usual definition of a monoidal model category. We refer the reader to \cite{Hov} for more details about the weak MM0 axiom. Here we use a stronger version because it is needed
in the construction of the model category of props in \cite{Fre}.
The category $Ch_{\mathbb{K}}$ of chain complexes over
a field $\mathbb{K}$ is our main working example of symmetric monoidal model category.
In the remaining part of the paper, we apply Definitions 2.1 and 2.3 to the case $\mathcal{C}=Ch_{\mathbb{K}}$.

\begin{defn}
A dg prop is a strict symmetric monoidal category $P$, enriched over $Ch_{\mathbb{K}}$,
with $\mathbb{N}$ as object set and the tensor product given by $m\otimes n=m+n$ on objects.
\end{defn}
\begin{rem}
A differential graded (dg) $\Sigma$-biobject is a sequence $M=\{M(m,n)\}_{m,n\in\mathbb{N}}$
of chain complexes such that each $M(m,n)$ is endowed with a left action of the symmetric group $\Sigma_m$ and
a right action of the symmetric group $\Sigma_n$ commuting with the left one.
We can see $M(m,n)$ as a space of operations with $m$ inputs and $n$ outputs, and the action of the symmetric groups
as permutations of the inputs and the outputs. We call $(m,n)$ the biarity of such an operation.

Composing operations of two $\Sigma$-biobjects $M$ and $N$ amounts to considering $2$-levelled directed graphs
(with no loops) with the first level indexed by operations of $M$ and the second level by operations of $N$.
Vertical composition by grafting and horizontal composition by concatenation allows one to give the following
alternative definition of a prop.
A prop is a $\Sigma$-biobject equiped with horizontal products
\[
\circ_{h}:P(m_{1},n_{1})\otimes P(m_{2},n_{2})\rightarrow P(m_{1}+m_{2},n_{1}+n_{2}),
\]
vertical composition products
\[
\circ_{v}:P(k,n)\otimes P(m,k)\rightarrow P(m,n)
\]
and units $1\rightarrow P(n,n)$ corresponding to identity morphisms of the objects $n\in\mathbb{N}$ in $P$. These
operations satisfy relations coming from the axioms of symmetric monoidal
categories. We refer the reader to Enriquez and Etingof \cite{EE} for an explicit description
of props in the context of modules over a ring.

A morphism of props is then simply a morphism of $\Sigma$-biobjects, that is, a collection of equivariant morphisms,
preserving both types of compositions.
\end{rem}
We denote by $\mathcal{P}$ the category of props.

There exists a free prop functor $\mathcal{F}$ fitting an adjunction
\[
\mathcal{F}:\Sigma \rightleftarrows \mathcal{P} :U
\]
and having as right adjoint the forgetful functor $U$
between the category of props and the category of $\Sigma$-biobjects.
There is an explicit construction of the free prop for which we refer the reader to Section A.2 in the
appendix of \cite{Fre}.

For a prop $P$ in $Ch_{\mathbb{K}}$, we can define the notion of $P$-algebra in a symmetric monoidal category
over $Ch_{\mathbb{K}}$:

\begin{defn}
Let $\mathcal{E}$ be a symmetric monoidal category over $Ch_{\mathbb{K}}$.

(1) The endomorphism prop of $X\in\mathcal{E}$ is given by $End_X(m,n)=Hom_{\mathcal{E}}(X^{\otimes m},X^{\otimes n})$
where $Hom_{\mathcal{E}}(-,-)$ is the external hom bifunctor of $\mathcal{E}$.

(2) Let $P$ be a prop in $Ch_{\mathbb{K}}$. A $P$-algebra in $\mathcal{E}$
is an object $X\in\mathcal{E}$ equipped with a prop morphism $P\rightarrow End_X$.
\end{defn}
\medskip{}

Now we are interested in doing homotopical algebra with props and their algebras.
Concerning props, we have to distinguish two cases. In the case where $\mathbb{K}$ is of characteristic zero,
we can use the adjunction $F:\Sigma\rightleftarrows\mathcal{P}:U$ to
transfer the cofibrantly generated model category structure on $\Sigma$-biobjects to the whole category of props
(see \cite{Fre}, Theorem 5.5). The weak equivalences and fibrations are the componentwise
quasi-isomorphisms and surjections, and the generating (acyclic) cofibrations are the images under the
free prop functor of the generating (acyclic) cofibrations of the category of diagrams $\Sigma$,
which we equip with the usual projective model structure (see \cite{Hir}, Theorem 11.6.1).

In the positive characteristic case, this construction works well only with the subcategory of
props with non-empty inputs (respectively, outputs), and does not give a full model structure,
but only a semi-model structure on this subcategory (see \cite{Fre}, Theorem 4.9 and Remark 4.10).
A prop $P$ has non-empty inputs if it satisfies $P(0,n)=\mathbb{K}$ if $n=0$
and $P(0,n)=0$ otherwise. The semi-model structure is close enough to a full model structure
to define the homotopy category and work similarly with homotopical algebra.
In the remaining part of the paper, we work in any of these settings.
In the positive characteristic case, we tacitly assume that our props satisfy the non-empty inputs requirement.

Concerning algebras over a dg prop, the situation is far more involved than in the operadic case.
In general, there is no model structure on such a category, there is even no limits or colimits
(and no free algebra functor). However, there are other ways to recover information of a homotopical nature
about algebras over a prop.
In order to get the full homotopy theory of any category with weak equivalences, one has to consider its
simplicial localization. This is the topic of the next section.

\section{Simplicial localization of algebras over cofibrant props}

Let $\mathcal{E}$ be a combinatorial symmetric monoidal model category over $Ch_{\mathbb{K}}$
and $(\mathcal{E}^c)$ its subcategory of cofibrant objects.
A combinatorial category is a cofibrantly generated model category which is also locally presentable,
see \cite{Dug} and the appendix of \cite{Lur} for the precise definition of this notion.
Let $(\mathcal{E}^c)^P$ denote the category of algebras over a dg prop $P$, and
$w(\mathcal{E}^c)^P$ its subcategory of morphisms of $P$-algebras which are weak equivalences in $\mathcal{E}$.

Let $\varphi:P\rightarrow Q$ be a morphism of dg props.
For any $Q$-algebra $X$, that is, an object $X$ of $\mathcal{E}$ equipped with a dg prop morphism $\phi_X:Q\rightarrow End_X$,
we can define a $P$-algebra having the same underlying object $X$ with the $P$-algebra structure
$\phi_X\circ\varphi:P\rightarrow Q\rightarrow End_X$.
\begin{lem}
This construction defines a functor $\varphi^*:(\mathcal{E}^c)^Q\rightarrow (\mathcal{E}^c)^P$ which restricts
to a functor $w(\mathcal{E}^c)^Q\rightarrow w(\mathcal{E}^c)^P$.
\end{lem}
\begin{proof}
A morphism of $Q$-algebras can be encoded by the data of a morphism $f:X\rightarrow Y$ of $\mathcal{E}$
and a prop morphism $\phi_f:Q\rightarrow End_f$, where $End_f$ is the endomorphism prop of $f$, whose construction is given
by a pullback in $\Sigma$-biobjects
\[
\xymatrix{
End_f \ar[r] \ar[d] & End_X \ar[d]^{f_*} \\
End_Y \ar[r]_{f^*} & Hom_{XY}
}
\]
where $Hom_{XY}(m,n)=Hom(X^{\otimes m},Y^{\otimes n})$ (see \cite{Fre}, Proposition 7.1).
It is a particular case of the endomorphism prop of a diagram, for which we refer
the reader to \cite{Fre}, Section 6.3. We define $\varphi^*(f)$ as the $P$-algebra morphism consisting
of $f$ in the underlying category $\mathcal{E}$, equipped with the prop morphism $\phi_f\circ\varphi$.
This prop morphism is obviously compatible with $\phi_X$ and $\phi_Y$, and sends the identities
to the identities. We just have to check the compatibility with the composition.
Let $f:X\rightarrow Y$ and $g:Y\rightarrow Z$ be two morphisms of $Q$-algebras, with associated prop morphisms
$\phi_f:Q\rightarrow End_f$, respectively $\phi_g:Q\rightarrow End_g$. By Proposition 7.1 of \cite{Fre},
the projections $p^f:End_f\rightarrow End_Y$ and $p^g:End_g\rightarrow End_Y$ are prop morphisms, so we can consider
their pullback $End_f\times_{End_Y} End_g$ in the category of props. This pullback comes with a prop morphism
$End_f\times_{End_Y} End_g\rightarrow End_{g\circ f}$. The commutative square
\[
\xymatrix{
Q \ar[r]^{\phi_g} \ar[d]_{\phi_f} & End_g \ar[d]^{p^g} \\
End_f \ar[r]_{p^f} & End_Y
}
\]
induces a morphism $(\phi_f,\phi_g):Q \rightarrow End_f\times_{End_Y} End_g$, hence a morphism
$\phi_{g\circ f}:Q\rightarrow End_{g\circ f}$ which is the $Q$-algebra structure on $g\circ f$ induced by the $Q$-algebra
structures on $f$ and $g$. The morphism $\phi_{g\circ f}\circ\varphi$ is then obtained similarly
by composing $(\varphi\circ\phi_f,\varphi\circ\phi_g)$ with $End_f\times_{End_Y} End_g\rightarrow End_{g\circ f}$.
Thus we finally get $\varphi^*(g\circ f)=\varphi^*(g)\circ\varphi^*(f)$.
Since $\varphi^*$ does not change the underlying objects and morphisms of $\mathcal{E}$, it obviously
restricts to a functor $w(\mathcal{E}^c)^Q\rightarrow w(\mathcal{E}^c)^P$.
\end{proof}
The goal of this section is to prove the following homotopy invariance result,
which implies Theorem 1.1 as a special case:
\begin{thm}
A weak equivalence $\varphi:P\stackrel{\sim}{\rightarrow}Q$ between two cofibrant dg (colored) props
induces a Dwyer-Kan equivalence between the associated simplicial localizations
\[
L(\varphi^*):L^H((\mathcal{E}^c)^Q,w(\mathcal{E}^c)^Q)\stackrel{\sim}{\rightarrow} L((\mathcal{E}^c)^P,w(\mathcal{E}^c)^P).
\]
\end{thm}
\begin{rem}
In Theorem 1.1 we consider a functor defined for $P$-algebras in $Ch_{\mathbb{K}}$, here we extend this functor
to categories tensored and enriched over $Ch_{\mathbb{K}}$. This extension is actually straightforward.
Indeed, recall that a $Q$-algebra is the data of an object $X$ of $\mathcal{E}$ and a prop morphism $Q\rightarrow End_X$.
The functor $\varphi^*$ is defined by precomposing this morphism with $\varphi:P\rightarrow Q$, hence it is well
defined for any category $\mathcal{E}$ tensored and enriched over $Ch_{\mathbb{K}}$ (one just uses the endomorphism
prop defined by the external dg hom of $\mathcal{E}$).
\end{rem}

The proof is divided into three steps. Each step consists roughly in reducing the equivalence of Theorem 3.2,
by rewriting it in terms of different models of $(\infty,1)$-categories
until we abut to a comparison problem of bisimplicial sets.
We solve this problem by using an improved version of the homotopy invariance theorem of classification spaces
obtained in \cite{Yal1}.

Here we use three models of $(\infty,1)$-categories: complete Segal spaces (\cite{Rez2}), simplicial categories (\cite{DK1},
\cite{DK2},\cite{DK3}, \cite{Ber2})
and relative categories (\cite{BK1},\cite{BK2}). We will not enter in the details of these theories but just recall
informally, when it will be necessary, the definitions and properties we need.

\textbf{Set theoretic warning.} The construction of these models of $(\infty,1)$-categories
is a priori established in the setting of small categories,
in order to avoid set theoretic problems like simplicial proper classes instead of sets or categories which
are not locally small. However in practice one often wants to apply these results to large categories.
We adopt therefore Grothendieck's axiom of universes to sort out this issue:
for every set there exists a universe in which this set is an element. Thus there exists a universe $U$
in which the categories $(\mathcal{E}^c)^Q$ and $(\mathcal{E}^c)^P$ are $U$-small.

\subsection{Step 1. From simplicial categories to relative categories}

A simplicial category is a category enriched over simplicial sets. We denote by $SCat$ the category of simplicial categories.
There exists functorial cosimplicial resolutions and simplicial resolutions in any model category (\cite{DK3},\cite{Hir}),
so model categories provide examples of (weakly) simplicially enriched categories. One recovers the morphisms of the homotopy category from a cofibrant object to a fibrant object by taking the set of connected components of the corresponding
simplicial mapping space. Another more general example is the simplicial localization developed by Dwyer and Kan \cite{DK1}.
Let $(C,W)$ be a pair of categories such that $W$ is a subcategory of $C$ containing all the objects of $C$. We call $W$
the category of weak equivalences of $C$. Such a pair is called a relative category in \cite{BK1}. To any relative category
Dwyer and Kan associates a simplicial category $L(C,W)$ called its simplicial localization. They developed also another
simplicial localization, the hammock localization $L^H(C,W)$ \cite{DK2}. The two are actually equivalent in a sense we
are going to precise below, and each one has
its own advantages. The hammock localization satisfies certain properties of the simplicial localization only up
to homotopy, but possesses the great advantage of having an explicit description of its simplicial mapping spaces.
By taking the sets of connected components of the mapping spaces, we get $\pi_0L(C,W)\cong C[W^{-1}]$ where
$C[W^{-1}]$ is the localization of $C$ with respect to $W$ (i.e. the homotopy category of $(C,W)$).

Let us define Dwyer-Kan equivalences:
\begin{defn}
Let $C$ and $D$ be two simplicial categories. A functor $f:C\rightarrow D$ is a Dwyer-Kan equivalence
if it induces weak equivalences of simplicial sets $Map_C(X,Y)\stackrel{\sim}{\rightarrow}Map_D(FX,FY)$
for every $X,Y\in C$, as well as inducing
an equivalence of categories $\pi_0C\stackrel{\sim}{\rightarrow}\pi_0D$.
\end{defn}
In particular, every Quillen equivalence of model categories gives rise to a Dwyer-Kan equivalence of their
simplicial localizations, as well as a Dwyer-Kan equivalence of their hammock localizations
(see \cite{DK3} Proposition 5.4).
The simplicial and hammock localizations are equivalent in the following sense:
\begin{prop}(Dwyer-Kan \cite{DK2}, Proposition 2.2)
Let $(C,W)$ be a relative category.
There is a zigzag of Dwyer-Kan equivalences
\[
L^H(C,W)\leftarrow diag L^H(F_*C,F_*W)\rightarrow L(C,W)
\]
where $F_*C$ is a simplicial category called the standard resolution of $C$ (see \cite{DK1} Section 2.5).
\end{prop}
\begin{rem}
This implies that Theorems 1.1 and 3.2 can be equivalently enunciated with the hammock localization $L^H(\varphi^*)$ or
with the simplicial localization $L(\varphi^*)$. Since proofs of \cite{BK2} are stated with the hammock localization,
we will do the same in our proof.
\end{rem}
By Theorem 1.1 of \cite{Ber2}, there exists a model category structure on the category of (small) simplicial categories with the Dwyer-Kan equivalences as weak equivalences.
Since every simplicial category is Dwyer-Kan equivalent to the simplicial localization of a certain relative category
(see for instance \cite{BK2}, Theorem 1.7),
this model structure forms a homotopy theory of homotopy theories.

Let us denote by $RelCat$ the category of relative categories. The objects are the relative categories
and the morphisms are the relative functors, that is, the functors restricting to functors between the
categories of weak equivalences. Lemma 3.1 tells us that $\varphi^*$ is  a well defined relative functor.

By Theorem 6.1 of \cite{BK1}, there is an adjunction between
the category of bisimplicial sets and the category of relative categories
\[
K_{\xi}:sSets^{\Delta^{op}}\leftrightarrows RelCat:N_{\xi}
\]
(where $K_{\xi}$ is the left adjoint and $N_{\xi}$ the right adjoint)
which lifts any Bousfield localization of the Reedy model structure of bisimplicial sets into a model structure on $RelCat$.
In the particular case of the Bousfield localization defining the complete Segal spaces \cite{Rez2}, one obtains
a Quillen equivalent homotopy theory of the homotopy theories in $RelCat$ \cite{BK1}.
The weak equivalences of $RelCat$ will be called, following \cite{BK1}, Rezk equivalences.
We refer the reader to Section 5.3 of \cite{BK1} for the definition of the functor $N_{\xi}$.

Recall that we want to get a Dwyer-Kan equivalence
\[
L^H(\varphi^*):L^H((\mathcal{E}^c)^Q,w(\mathcal{E}^c)^Q)\stackrel{\sim}{\rightarrow} L^H((\mathcal{E}^c)^P,w(\mathcal{E}^c)^P).
\]
According to \cite{BK2}, Theorem 1.8, a morphism of relative categories is a Rezk equivalence if and only if it induces a Dwyer-Kan equivalence of the associated hammock localizations. So we have:
\begin{prop}
The simplicial functor
\[
L^H(\varphi^*):L^H((\mathcal{E}^c)^Q,w(\mathcal{E}^c)^Q)\stackrel{\sim}{\rightarrow} L^H((\mathcal{E}^c)^P,w(\mathcal{E}^c)^P).
\]
is a Dwyer-Kan equivalence if and only if if and only the relative functor
\[
\varphi^*:((\mathcal{E}^c)^Q,w(\mathcal{E}^c)^Q)\stackrel{\sim}{\rightarrow}
((\mathcal{E}^c)^P,w(\mathcal{E}^c)^P)
\]
is a Rezk equivalence.
\end{prop}

\subsection{Step 2. From relative categories to complete Segal spaces}

Although simplicial categories are the most intuitive model for $(\infty,1)$-categories, Dwyer-Kan equivalences are
difficult to detect. A nice model with weak equivalences easier to handle has been developed by Rezk in \cite{Rez2},
namely the category of complete Segal spaces. We denote this category by $CSS$. It has a model structure defined by a certain left Bousfield localization of the standard Reedy model structure on bisimplicial sets (Theorem 7.2 of \cite{Rez2}). The fibrant objects of $CSS$
are precisely the complete Segal spaces. To each complete Segal space $S$ one can associate the homotopy theory of a certain
relative category: the objects are the $0$-simplices of $S_0$, and for $x,y\in S_{0,0}$ the mapping space $map_S(x,y)$
is the fiber of the product of faces $(d_0,d_1):S_1\rightarrow S_0\times S_0$ over $(x,y)$. Equivalences are then
the points corresponding to invertible components in $\pi_0 map_S(x,y)$. We refer the reader to Section 5 of \cite{Rez2}
for the detailed construction.
The Reedy weak equivalences between two complete Segal spaces are precisely the Dwyer-Kan equivalences between
their associated homotopy theories (Theorem 7.2 of \cite{Rez2}).
An important instance of complete Segal space is the Reedy fibrant resolution of the classification diagram of a simplicial model category (see Theorem 8.3 of \cite{Rez2}). The homotopy category of this space is equivalent to
the homotopy category of the model category in the usual sense.
Thus a Reedy weak equivalence of classification diagrams of simplicial model categories
corresponds to an equivalence of their homotopy categories.
The notion of classification diagram will be defined in Definition 3.9.

Recall that, following Section 5.1 and Theorem 6.1 of \cite{BK1}, the functor $N_{\xi}$
defines the weak equivalences in $RelCat$, that is, a morphism of relative categories is a Rezk equivalence
if and only if its image under $N_{\xi}$ is a weak equivalence in $CSS$.
\begin{prop}
The relative functor
\[
\varphi^*:((\mathcal{E}^c)^Q,w(\mathcal{E}^c)^Q)\stackrel{\sim}{\rightarrow}
((\mathcal{E}^c)^P,w(\mathcal{E}^c)^P)
\]
is a Rezk equivalence if and only if the morphism of bisimplicial sets
\[
N_{\xi}(\varphi^*):RN_{\xi}((\mathcal{E}^c)^Q,w(\mathcal{E}^c)^Q)\stackrel{\sim}{\rightarrow}
N_{\xi}((\mathcal{E}^c)^P,w(\mathcal{E}^c)^P)
\]
is a weak equivalence in $CSS$.
\end{prop}

\subsection{Step 3. Comparison of classification diagrams}

Let us denote by $sSets$ the category of simplicial sets.
\begin{defn}
Let $(C,W)$ be a relative category. Its classification diagram, denoted by $N(C,W)$, is the bisimplicial set
$N(C,W):\Delta^{op}\rightarrow sSets$ defined by
\[
N(C,W)([n])=\mathcal{N}we(C^{[n]})
\]
where $C^{[n]}$ is the category of diagrams over $[n] = {0\rightarrow...\rightarrow n}$.
We write $we(C^{[n]})$ for the subcategory defined by the (pointwise) weak-equivalences in this diagram category,
and $\mathcal{N}$ refers to the simplicial nerve functor (see \cite{MLa2}, Chapter XII, Section 2, for a definition
of this nerve).
\end{defn}
The simplicial set $N(C,W)([0])=\mathcal{N}W$ is the classification space of $(C,W)$.
The classification diagram construction gives rise to a functor $N:RelCat\rightarrow CSS$.

By Lemma 5.4 of \cite{BK1}, there is a natural Reedy weak equivalence
\[
N\stackrel{\sim}{\rightarrow} N_{\xi}.
\]
This implies, for any prop morphism $\varphi:P\rightarrow Q$, the existence of a commutative diagram
of bisimplicial sets
\[
\xymatrix{
N((\mathcal{E}^c)^Q,w(\mathcal{E}^c)^Q) \ar[r]^{N(\varphi^*)} \ar[d]_{\sim} & N((\mathcal{E}^c)^P,w(\mathcal{E}^c)^P)
\ar[d]_{\sim}\\
N_{\xi}((\mathcal{E}^c)^Q,w(\mathcal{E}^c)^Q) \ar[r]_{N_{\xi}(\varphi^*)} & N_{\xi}((\mathcal{E}^c)^P,w(\mathcal{E}^c)^P)
}
\]
where the vertical arrows are Reedy weak equivalences.
We have the following homotopy invariance result:
\begin{thm}
A weak equivalence of cofibrant dg props $\varphi:P\stackrel{\sim}{\rightarrow}Q$ induces a Reedy weak equivalence
of classification diagrams
\[
N(\varphi^*):N((\mathcal{E}^c)^Q,w(\mathcal{E}^c)^Q)\stackrel{\sim}{\rightarrow} N((\mathcal{E}^c)^P,w(\mathcal{E}^c)^P).
\]
\end{thm}
By Theorem 3.10 and the commutative square above, for any weak equivalence of cofibrant dg props $\varphi:P\stackrel{\sim}{\rightarrow}Q$,
the map $N_{\xi}(\varphi^*)$ is a Reedy weak equivalence of bisimplicial
sets. Given that $CSS$ is obtained as a left Bousfield localization of the Reedy model structure on bisimplicial sets,
Reedy weak equivalences are still weak equivalences in $CSS$ so $N_{\xi}(\varphi^*)$ is a weak equivalence
in $CSS$. By Propositions 3.7 and 3.8 this implies that $L^h(\varphi^*)$ is a Dwyer-Kan equivalence, hence concludes
the proof of Theorem 3.2.

To prove Theorem 3.10 we use the following results:
\begin{prop}[(Lurie, proposition A.2.8.2 of {\cite{Lur}})]
Let us suppose that $\mathcal{E}$ is combinatorial. Then the injective model structure on the category of diagrams $\mathcal{E}^I$,
with point-wise weak-equivalences and cofibrations, is combinatorial (in particular, it is cofibrantly generated).
\end{prop}

\begin{prop}
The model category $(\mathcal{E}^I)_{inj}$ is a symmetric monoidal model category
over $Ch_{\mathbb{K}}$.
\end{prop}
To prove this we use the following transitivity lemma:
\begin{lem}
Let $\mathcal{E}$ be a symmetric monoidal model category over $\mathcal{C}$ and $\mathcal{D}$ be a
symmetric monoidal model category over $\mathcal{E}$. Then $\mathcal{D}$ is a symmetric monoidal model
category over $\mathcal{C}$.
\end{lem}
\begin{proof}
We have a symmetric monoidal functor $\eta_{\mathcal{E}}:\mathcal{C}\rightarrow\mathcal{E}$
defined by $\eta_{\mathcal{E}}(C)=C\otimes_{\mathcal{C}}^{\mathcal{E}} 1_{\mathcal{E}}$ where $\otimes_{\mathcal{C}}^{\mathcal{E}}$
is the external tensor product of $\mathcal{E}$ over $\mathcal{C}$ and $1_{\mathcal{E}}$ is the unit of $\mathcal{E}$.
According to axiom MM0 in $\mathcal{E}$, the unit $1_{\mathcal{E}}$ is cofibrant, so  by applying axiom MM1
for the external tensor product we see that $\eta_{\mathcal{E}}$ preserves cofibrations and acyclic cofibrations.

Now, let $\otimes_{\mathcal{E}}^{\mathcal{D}}$ denote the external tensor product of $\mathcal{D}$ over $\mathcal{E}$.
We can define an external tensor product of $\mathcal{D}$ over $\mathcal{C}$ by
\[
C\otimes_{\mathcal{C}}^{\mathcal{D}} D=\eta_{\mathcal{E}}(C)\otimes_{\mathcal{E}}^{\mathcal{D}} D
\]
where $C\in\mathcal{C}$ and $D\in\mathcal{D}$. It satisfies all the recquired axioms of an external tensor product
since $\otimes_{\mathcal{C}}^{\mathcal{E}}$ is an external tensor product and $\eta_{\mathcal{E}}$ is a
symmetric monoidal functor. The external tensor product $\otimes_{\mathcal{E}}^{\mathcal{D}}$ satisfies MM1
and $\eta_{\mathcal{E}}$ preserves cofibrations and acyclic cofibrations, so $\otimes_{\mathcal{C}}^{\mathcal{D}}$
satisfies also MM1.
\end{proof}

\begin{proof}[Proof of Proposition 3.12]
First we have to prove that
the model category $(\mathcal{E}^I)_{inj}$ is a symmetric monoidal model category
for the pointwise tensor product. The pushout-product is a pointwise pushout-product since colimits are created pointwise, and (acyclic) cofibrations are the pointwise (acyclic) cofibrations by definition, so one just has to apply axiom MM1 of $\mathcal{E}$.

Then we claim that $(\mathcal{E}^I)_{inj}$ is a symmetric monoidal model category over $\mathcal{E}$
for the external tensor product defined for every $E\in \mathcal{E}$ and $F\in \mathcal{E}^I$ by
\[
\forall i\in I, (E\otimes_e F)(i)=E\otimes F(i).
\]
To see this,
suppose $f$ is a cofibration of $\mathcal{E}$ and $\phi$ be a cofibration of $(\mathcal{E}^I)_{inj}$.
We form the pushout-product $f\Box \phi$. By definition of the external tensor product, for every $i\in I$
the map $(f\Box \phi)(i)=f\Box \phi(i)$ is a pushout-product of $f$ and $\phi(i)$ in $\mathcal{E}$.
We apply axiom MM1 in $\mathcal{E}$ and use the fact that cofibrations are defined pointwise in
$(\mathcal{E}^I)_{inj}$, so $f\Box \phi$ is a cofibration. If one of these two maps is acyclic
then $f\Box \phi$ is acyclic by the same argument.

Finally we apply the transitivity lemma to $(\mathcal{E}^I)_{inj}$, $\mathcal{E}$ and $Ch_{\mathbb{K}}$ to conclude the proof of this proposition.
\end{proof}

We now need the following homotopy invariance result:
\begin{thm}(Yalin, Theorem 2.12 of \cite{Yal1})
Let $\mathcal{E}$ be a cofibrantly generated symmetric monoidal model
category over $Ch_{\mathbb{K}}$. Let $\varphi:P\stackrel{{\sim}}{\rightarrow}Q$
be a weak equivalence between two cofibrant props defined in $Ch_{\mathbb{K}}$.
This morphism $\varphi$ gives rise to a functor $\varphi^{*}: w(\mathcal{E}^{c})^{Q}\rightarrow w(\mathcal{E}^{c})^{P}$
which induces a weak equivalence of simplicial sets $\mathcal{N}\varphi^{*}:\mathcal{N} w(\mathcal{E}^{c})^{Q}\rightarrow\mathcal{N} w(\mathcal{E}^{c})^{P}$.
\end{thm}
\begin{rem}
If we suppose that $\mathbb{K}$ is of characteristic zero, then this result can be extended to colored props,
i.e. props equiped with a set of colors labeling inputs and outputs of the operations (see Section 3 of \cite{Yal1}).
\end{rem}

\begin{proof}[Proof of Theorem 3.10]
Let $\mathcal{E}$ be a combinatorial symmetric monoidal model category over $Ch_{\mathbb{K}}$.
Let $I$ be a small category and $\mathcal{E}^I$ the category of $I$-diagrams in $\mathcal{E}$.
We equip $\mathcal{E}^I$ with the standard injective model structure.

In the injective model structure, cofibrant diagrams are diagrams of cofibrant objects of $\mathcal{E}$, that is
$((\mathcal{E}^I)_{inj})^c=((\mathcal{E}^c)^I)_{inj}$.
We apply Theorem 3.14 to $(\mathcal{E}^I)_{inj}$.
A weak equivalence of cofibrant dg props $\varphi:P\stackrel{\sim}{\rightarrow}Q$ induces a weak equivalence
of classification spaces
\[
\mathcal{N}w((\mathcal{E}^c)^Q)^I=\mathcal{N}w((\mathcal{E}^c)^I)^Q\stackrel{\sim}{\rightarrow}
\mathcal{N}w((\mathcal{E}^c)^I)^P=\mathcal{N}w((\mathcal{E}^c)^P)^I,
\]
in particular for $I=[n]$ and for every $n\geq 0$, giving rise to the desired Reedy weak equivalence of
bisimplicial sets.
\end{proof}

\begin{cor}
Let $P$ be either any dg prop over a field of characteristic zero, or a prop with non-empty inputs over a field
of any characteristic.
The homotopy theory of homotopy $P$-algebras does not depend, up to Dwyer-Kan equivalence, on the choice of a cofibrant resolution of $P$.
\end{cor}
Recall that according to Theorem 1.1 of \cite{JY}, the model category structure on props can be extended to colored props when we work over a field of characteristic zero, so we also get:
\begin{cor}
Let $P$ be a differential graded colored prop over a field of characteristic zero.
The homotopy theory of homotopy $P$-algebras does not depend, up to Dwyer-Kan equivalence, on the choice of a cofibrant resolution of $P$.
\end{cor}

Examples in the $1$-colored case include the list of generalized bialgebras described in \cite{Lod}, for instance Hopf algebras, non-commutative Hopf algebras, Lie bialgebras, etc. These bialgebras are determined by a distributive law between operations and cooperations.
Other examples of bialgebras which are not of this kind include Frobenius bialgebras and more general topological field theories.
Typical examples of colored props are the one encoding diagrams of algebras (see \cite{FMY}), so this is an invariance result in particular for homotopy diagrams of various kinds of  homotopy algebras.

\end{document}